\documentclass{amsart}
\usepackage{amsmath, amsthm,hyperref}
\usepackage{color}
\allowdisplaybreaks

\theoremstyle{theorem}
\newtheorem{theorem}{Theorem}
\newtheorem{corollary}[theorem]{Corollary}
\newtheorem{prop}[theorem]{Proposition}
\newtheorem{lemma}[theorem]{Lemma}
\newtheorem{definition}{Definition}

\DeclareMathOperator{\mex}{mex}

\DeclareMathOperator{\crank}{crank}

\title{Dyson's Crank and the Mex of Integer Partitions}
\author{Brian Hopkins}
\address{Saint Peter's University, Jersey City NJ 07306}
\email{bhopkins@saintpeters.edu}
\author{James A. Sellers}
\address{University of Minnesota Duluth, Duluth MN 55812}
\email{jsellers@d.umn.edu}
\author{Dennis Stanton}
\address{University of Minnesota, Minneapolis MN 55455}
\email{stanton@math.umn.edu}


\begin{document}
\maketitle

\begin{abstract}
Andrews and Newman have recently introduced the notion of the mex of a partition, the smallest positive integer that is not a part.  The concept has been used since at least 2011, though, with connections to Frobenius symbols.  Recently the parity of the mex has been associated to the crank statistic named by Dyson in 1944.  In this note, we extend and strengthen the connection between the crank and mex (along with a new generaliztion of the mex) by proving a number of properties that naturally relate these partition statistics.
\end{abstract}

\section{Introduction}
Given a positive integer $n$, a partition of $n$ is a collection of positive integers $\lambda = (\lambda_1, \ldots, \lambda_j)$ with $\sum \lambda_i = n$.  The $\lambda_i$, called parts, are ordered so that $\lambda_1 \ge \cdots \ge \lambda_j$.  Write $p(n)$ for the number of partitions of $n$.

The crank is a partition statistic requested by Dyson \cite{d} in 1944 
to combinatorially prove a congruence of $p(n)$ values found by Ramanujan using analytic methods.  The interested reader can learn more about the crank, including its definition, from the 1988 work of Andrews and Garvan \cite{ag}.  The primary tool that we need in this paper is the generating function for the crank.  
Let $M(m,n)$ denote the number of partitions of $n$ with crank $m$.  
We use the $q$-series notation \[(a;q)_n = (1-aq)\cdots(1-aq^n) \text{ and } (a;q)_\infty = \lim_{n \rightarrow \infty} (a;q)_n.\]  When $a=q$, we write $(q)_n$ and $(q)_\infty$ instead.  The following is \cite[Theorem 7.19]{g}.
\begin{theorem}[Garvan]
For a given integer $m$, the generating function for the number of partitions of $n$ with crank $m$ is
\begin{equation} \label{cgf}
\sum_{n=0}^\infty M(m,n) q^n = \frac{1}{(q)_\infty} \sum_{n=1}^\infty (-1)^{n-1} q^{n(n-1)/2 + n|m|} (1-q^n). 
\end{equation}
\end{theorem}

While introducing work of Karl Mahlburg, George Andrews and Ken Ono \cite{ao} discussed ``the central role of the crank in the theory of partitions.''  Our goal is to further elucidate Dyson's crank by way of a newer partition statistic, the mex. 

The minimal excludant of a set of integers, abbreviated mex, is the smallest positive integer not in the set.  Some examples of the mex applied to partitions follow. 
\[\mex(3,2) = 1, \quad \mex(3,1,1) = 2, \quad \mex(2,2,1) = 3.\]
Andrews first used ``the smallest part that is \textit{not} a summand'' in 2011 \cite{a}; he and Newman started using the term mex, from combinatorial game theory, in 2019 \cite{an19}.

Here is an outline of the remainder of the paper.  In Section \ref{s2}, we introduce a generalization of the mex statistic and connect it to the partitions with a given minimum crank.  This, along with a formula for the number of partitions with a given mex, leads to a formula for the number of partitions with a given crank.
Section \ref{s3} establishes connections between the crank and a classic partition representation, the Frobenius symbol.
Finally, in Section \ref{s4}, we refine the partitions with odd mex by considering the mex modulo 4; this unveils a surprising identity in which the function $q(n)$, the number of partitions of $n$ with distinct parts, arises.

\section{Crank and a generalized mex} \label{s2}
We introduce a generalization of the mex in order to find results on $M(m,n)$ for arbitrary $m$.  
We will focus on cases with $m \ge 0$.  Results for negative $m$ follow from the relation $M(m,n) = M(-m,n)$ \cite[(1.23)]{g}.

\begin{definition}
Given a partition $\lambda$ and a positive integer $j$ that is a part of $\lambda$, define $\mex_j(\lambda)$ to be the least integer greater than $j$ that is not a part of $\lambda$.  If $j$ is not among the parts of $\lambda$, then $\mex_j(\lambda)$ is not defined.
\end{definition}

If we consider 0 to be a part of every partition, then $\mex_j$ for $j=0$ reduces to the mex described above.

Our first result gives the number of partitions with a specified minimum crank in terms of $\mex_j$.  

\begin{theorem} \label{jcrank}
Given a nonnegative integer $j$, the number of partitions of $n$ with $\crank(\lambda) \ge j$ equals the number of partitions of $n$ such that $\mex_j(\lambda) - j$ is odd.
\end{theorem}

The core of the proof is in modifying \eqref{cgf} to address partitions with a specified minimum crank.  We state this as a lemma since it leads to additional results.

\begin{lemma} \label{crankj+}
For a given integer $j$, the generating function for the number of partitions of $n$ with $\crank(\lambda) \ge j$ is
\[ \frac{1}{(q)_\infty} \sum_{k=0}^\infty q^{j(2k+1) + k(2k+1)} (1-q^{2k+j+1}). \]
\end{lemma}
\begin{proof}
Summing \eqref{cgf} over $m\geq j$ yields  
\begin{align}
\sum_{m= j}^\infty \sum_{n=0}^\infty M(m,n) q^n & =  \frac{1}{(q)_\infty} \sum_{n=1}^\infty (-1)^{n-1} q^{n(n-1)/2} (1-q^n) \sum_{m= j}^\infty q^{nm} \nonumber \\
& = \frac{1}{(q)_\infty} \sum_{n=1}^\infty (-1)^{n-1} q^{n(n-1)/2} (1-q^n) \left( \frac{q^{nj}}{1-q^n} \right) \nonumber \\
& = \frac{1}{(q)_\infty} \sum_{n=1}^\infty (-1)^{n-1} q^{n(n-1)/2}q^{nj} \label{later} \\
& = \frac{1}{(q)_\infty} \sum_{n=1}^\infty (-1)^{n-1} q^{n(n-1+2j)/2} \nonumber \\
& = \frac{1}{(q)_\infty} \left( \sum_{k= 1}^\infty q^{(2k-1)((2k-1)-1+2j)/2} - \sum_{k= 1}^\infty q^{2k(2k-1+2j)/2} \right) \nonumber \\
& = \frac{1}{(q)_\infty} \left( \sum_{k= 0}^\infty q^{(2k+1)(k+j)} - \sum_{k= 0}^\infty q^{(k+1)(2k+1+2j)} \right) \nonumber \\
& = \frac{1}{(q)_\infty} \sum_{k=0}^\infty q^{j(2k+1)+k(2k+1)} (1 - q^{2k+j+1}).  \nonumber \qedhere
\end{align}
\end{proof}

\begin{proof}[Proof of Theorem \ref{jcrank}.]
The partition condition $\mex_j - j$ odd means, for some nonnegative integer $k$, that $j, j+1, \dots, 2k+j$ are parts and $2k+j+1$ is not.  The generating function for such partitions of $n$ is 
\begin{align*}
 \sum_{k=0}^\infty \frac{q^{j+(j+1)+\cdots+(j+2k)}}{\prod_{i=1, \, i \ne 2k+j+1}^\infty (1-q^i)}   
 & = \frac{1}{(q)_\infty} \sum_{k=0}^\infty q^{j(2k+1) + (2k)(2k+1)/2} (1-q^{2k+j+1})  \\
 & = \frac{1}{(q)_\infty} \sum_{k=0}^\infty q^{j(2k+1) + k(2k+1)} (1-q^{2k+j+1}) 
 \end{align*}
which matches the generating function in Lemma \ref{crankj+} for the number of partitions of $n$ with $\crank(\lambda) \ge j$.
\end{proof}

Let $o(n)$ be the number of partitions of $n$ with odd mex and $e(n)$ the number with even mex.  The $j=0$ case of Theorem \ref{jcrank}, i.e., that $o(n)$ equals the number of partitions of $n$ with nonnegative crank, was proven independently by Andrews--Newman \cite[Thm. 2]{an20} and Hopkins--Sellers \cite[Thm. 1]{hs} in 2020.  For the $j=1$ case, the conditions of Theorem \ref{jcrank} are equivalent to having even mex, thus $e(n)$ equals the number of partitions of $n$ with positive crank \cite[Cor. 2]{hs}.

Next, we derive a formula for the number of partitions with a given crank in terms of $p(n)$.  We will then connect a special case of this with the mex.

\begin{corollary} \label{crankrecur}
Given nonnegative $j$, the number of partitions of $n$ with crank $j$ is
\begin{align*}
M(j,n) & = \sum_{k=1}^\infty (-1)^{k+1} \left(p\!\left( n - \frac{k(k+2j-1)}{2} \right) - p\!\left( n - \frac{k(k+2j+1)}{2}\right)\right) \\
	& = p(n-j) - p(n-j-1) - p(n-2j-1) + p(n-2j-3) + \cdots.
\end{align*}
\end{corollary}

Note that such a formula involves finitely many terms since $p(n) = 0$ for $n<0$.

\begin{proof}
Using \eqref{later} and interpreting $1/(q)_\infty$ as $\sum p(n) q^n$, the number of partitions of $n$ with minimum crank $j$ is
\[\sum_{k=1}^\infty (-1)^{k+1} p\!\left( n - \binom{k}{2}-kj \right).\]
Subtracting from this the analogous expression for the number of partitions of $n$ with minimum crank $j+1$ gives
\[M(j,n) = \sum_{k=1}^\infty (-1)^{k+1} \left(p\!\left( n - \binom{k}{2}-kj \right) - p\!\left( n - \binom{k}{2}-k(j+1) \right)\right)\]
which is equivalent to the stated result. 
\end{proof}

One could also prove Corollary \ref{crankrecur} from \eqref{cgf}.
 
We want to highlight the $m=0$ case of Corollary \ref{crankrecur}, which is Corollary \ref{0crank}.  While it needs no additional proof, we provide a second verification in order to apply the following result which we will use again in the final section.

Let $x(m,n)$ be the number of partitions of $n$ with mex $m$.  Write $t_k = 1 + \cdots + k$ for the $k$th triangular number.
\begin{prop} Given positive integers $n$ and $m$, the number of partitions of $n$ with mex $m$ is
\[x(m,n) = p(n-t_{m-1}) - p(n-t_m).\]
\label{mexform}
\end{prop}
\begin{proof} For a partition to have mex $m$, it must include $1, \ldots, m-1$ as parts and exclude $m$.  The number of partitions of $n$ with the required included parts is $p(n - t_{m-1})$, since removing one 1, one 2, \dots, one $m-1$ leaves a partition of $n-t_{m-1}$.  The number of partitions of $n$ with $m$ also included as a part is \[p(n-t_{m-1}-m) = p(n-t_m),\] so the number with $m$ excluded is $p(n - t_{m-1}) - p(n-t_m)$.
\end{proof}
One can find the ideas for this result in the second proof of \cite[Thm 1.1]{an19}.

We can use Proposition \ref{mexform} and the earlier results about odd and even mex to verify the formula for $M(0,n)$ in terms of $p(n)$.

\begin{corollary} \label{0crank}
The number of crank 0 partitions of $n$ is
\begin{align*}
M(0,n) & = p(n) + 2 \sum_{k=1}^\infty (-1)^k p\!\left(n-\frac{k(k+1)}{2}\right) \\
	& = p(n) - 2p(n-1) + 2p(n-3) - 2p(n-6) + \cdots.
\end{align*}
\end{corollary}
\begin{proof}
By the $j=0$ case of Theorem \ref{jcrank}, there are $o(n)$ partitions of $n$ with nonnegative crank.  As discussed above, there are $e(n)$ partitions of $n$ with positive crank.  Therefore,
\begin{align*}
M(0,n) & = o(n) - e(n) \\
		&=  \sum_{k=0}^\infty x(2k+1,n) - \sum_{k=0}^\infty x(2k+2,n) \\
		& = \sum_{k=0}^\infty (p(n-t_{2k}) - p(n-t_{2k+1})) - \sum_{k=0}^\infty (p(n-t_{2k+1}) - p(n-t_{2k+2})) \\
		& = \sum_{k=0}^\infty (p(n-t_{2k}) - 2p(n-t_{2k+1}) + p(n-t_{2k+2})) \\
		& = p(n) + 2 \sum_{k=1}^\infty (-1)^k p\!\left(n-\frac{k(k+1)}{2}\right). \qedhere
\end{align*}
\end{proof}

This formula for $M(0,n)$ was recorded by Vaclav Kotesovec in 2016 \cite[A064410]{o}.

In 2019, Yuefei Shen showed that $o(n) \ge e(n)$ for $n > 1$ \cite[Thm. 1]{s}.  Hopkins--Sellers improved this to $o(n) > e(n)$ for $n>2$ \cite[Cor. 3]{hs}.  Corollary \ref{0crank} provides a precise formula for $o(n) - e(n)$.

\section{Frobenius symbol and the crank} \label{s3}
Frobenius developed a two-row notation for partitions \cite{f} now called the Frobenius symbol.  Each row consists of strictly decreasing nonnegative integers; see \cite[p. 78]{ae} for more information.

In 2011, Andrews \cite[Thm. 4]{a} proved that the number of partitions of $n$ whose Frobenius symbol has no 0 in the top row equals the number of partitions of $n$ with odd mex (in the current terminology).  Thus, by Andrews--Newman \cite[Thm. 2]{an20} and Hopkins--Sellers \cite[Thm. 1]{hs}, the number of partitions of $n$ whose Frobenius symbol has no 0 in the top row equals the number of partitions of $n$ with nonnegative crank.  In this section, we extend this result of Andrews in two ways.

First, we connect the crank 0 partitions with partitions whose Frobenius symbols have no 0 in either row. Let $[q^k]f(q)$ denote the coefficient of $q^k$ in the polynomial $f(q)$.

\begin{prop}  \label{noF0}
The number of partitions of $n$ with crank 0 equals the number of partitions of $n$ whose Frobenius symbol has no 0 minus the number of partitions of $n-1$ whose Frobenius symbol has no 0.
\end{prop}
\begin{proof}
The generating function for Frobenius symbols with no occurrence of 0 is, letting $s$ be the number of parts,
\[ \sum_{s=0}^\infty q^s [t^s w^s] (-tq;q)_\infty (-wq;q)_\infty = \sum_{s=0}^\infty q^{s^2} \frac{q^{2s}}{(q)_s^2}.\]
Thus the number of partitions of $n$ whose Frobenius symbol has no 0 minus the number of partitions of $n-1$ of the same type is
\[\sum_{s=0}^\infty q^{s^2} \frac{q^{2s}}{(q)_s^2} - q\sum_{s=0}^\infty q^{s^2} \frac{q^{2s}}{(q)_s^2} 
= (1-q) \sum_{s=0}^\infty \frac{q^{s^2+2s}}{(q)_s^2}.\]
Now the $a=b=0$ case of Heine's transformation as given in \cite[Eq. III.3]{gr} with $c=q$ and $z=q^2$ gives
\[ \sum_{s=0}^\infty \frac{q^{s^2-s}(q^3)^s}{(q)_s^2} = (q^2;q)_\infty \sum_{k=0}^\infty \frac{q^{2k}}{(q)_k^2}.\]
Since $(q^2;q)_\infty = (q)_\infty/(1-q)$, we have
\[ (1-q) \sum_{s=0}^\infty \frac{q^{s^2+2s}}{(q)_s^2} = (q)_\infty \sum_{k=0}^\infty \frac{q^{2k}}{(q)_k^2}.\]
The right-hand side of this last equation is an alternative expression for the generating function for the number of crank 0 partitions of $n$ \cite[p. 5]{an20}.
\end{proof}

The sequence of the number of partitions of $n$ whose Frobenius symbol has no 0 is \cite[A188674]{o}.  By Proposition \ref{noF0}, the first differences of this sequence give the number of partitions with crank 0, \cite[A064410]{o}.

For our second result in this section, we connect partitions with minimal crank $j$ to certain partitions whose Frobenius symbols have no $j$ in the top row.  Andrews's \cite[Thm. 4]{a} is the $j=0$ case of the following theorem.  

\begin{theorem}
The number of partitions of $n$ with $\crank(\lambda) \ge j$ equals the number of partitions of $n - j$ whose Frobenius symbol has no $j$ in its top row.
\end{theorem}

\begin{proof}
The generating function for Frobenius symbols with no $j$ in the top row is, 
letting $s$ be the number of parts,
\begin{align*}
\sum_{s=0}^\infty q^s [t^s w^s] \frac{(-t;q)_\infty (-w;q)_\infty}{1+tq^j}
& =\sum_{s=0}^\infty q^s \frac{q^{\binom{s}{2}}}{(q)_s} \left(\sum_{b=0}^s (-1)^b q^{jb} \frac{q^{\binom{s-b}{2}}}{(q)_{s-b}}\right) \\
& =\sum_{s,b \ge 0} q^{s+b} \frac{q^{\binom{s+b}{2}}}{(q)_{s+b}} (-1)^b q^{jb} \frac{q^{\binom{s}{2}}}{(q)_s} \\
& = \sum_{b=0}^\infty q^{b+\binom{b}{2}} (-1)^b {q^{jb}} \left(\sum_{s=0}^\infty \frac{q^{s^2+bs}}{(q)_s (q)_{s+b}}\right) \\
& = \frac{1}{(q)_\infty} \sum_{b=0}^\infty (-1)^b q^{\binom{b+1}{2}}  {q^{jb}}
\end{align*}
where the last line follows from \cite[(4.1)]{a} based on the decomposition of all partitions by Durfee rectangles of size $s \times (s+b)$ \cite[p. 92]{gh}.
Next, by \eqref{later}, we know that the generating function for partitions with minimum crank $j$ is 
\begin{align*}
\frac{1}{(q)_\infty} \sum_{n=1}^\infty (-1)^{n-1} q^{n(n-1)/2}q^{nj}
& = \frac{1}{(q)_\infty} \sum_{n=0}^\infty (-1)^{n} q^{(n+1)n/2}q^{(n+1)j} \\
& = \frac{1}{(q)_\infty} \sum_{n=0}^\infty (-1)^{n} q^{\binom{n+1}{2}}q^{nj+j} \\
& = q^j \frac{1}{(q)_\infty} \sum_{n=0}^\infty (-1)^{n} q^{\binom{n+1}{2}}q^{nj}.
\end{align*}  
Comparing the corresponding coefficients in the two expressions gives the desired result.  
\end{proof}

\section{Splitting the odd mexes} \label{s4}
Recall that $o(n)$ is the number of partitions of $n$ with odd mex.  Write $o_1(n)$ for the number of partitions of $n$ with mex congruent to $1 \bmod 4$, similarly $o_3(n)$ for $3 \bmod 4$, so that $o(n) = o_1(n) + o_3(n)$.

Let $q(n)$ denote the number of partitions of $n$ with distinct parts.
We prove an interesting relationship between $o_1(n)$, $o_3(n)$, and $q(n)$ which depends on the parity of $n$.  

\begin{prop} \label{o13}
For $n \ge1$,
\[o_1(n) = \begin{cases} o_3(n) & \text{if $n$ is odd,} \\ o_3(n) + q(n/2) & \text{if $n$ is even.} \end{cases} \]
\label{comb}
\end{prop}

To prove Proposition \ref{o13}, we use the following 1973 result of John Ewell \cite[Thm. 2]{e}.  Recall that $t_k$ denotes the $k$th triangular number $1+\cdots+k$.  

\begin{theorem}[Ewell]
For each nonnegative integer $k$,
\begin{gather} 
\sum_{j=0}^\infty  (-1)^{t_j} p(2k-t_j)  = q(k), \label{eweven} \\
\sum_{j=0}^\infty  (-1)^{t_j} p(2k+1-t_j)  = 0. \label{ewodd}
\end{gather}
\end{theorem}

\begin{proof}[Proof of Proposition \ref{o13}]
From Proposition \ref{mexform},
\begin{align*}
o_1(n) & = \sum_{i=0}^\infty x(n,4i+1) =  \sum_{i=0}^\infty  (p(n-t_{4i}) - p(n-t_{4i+1})), \\
o_3(n) & = \sum_{i=0}^\infty x(n,4i+3) =  \sum_{i=0}^\infty  (p(n-t_{4i+2}) - p(n-t_{4i+3})). 
\end{align*}

Consider the two cases of the theorem separately.  First, suppose $n = 2k$.  Write \eqref{eweven} as
$$\sum_{i=0}^\infty  (p(2k-t_{4i}) - p(2k-t_{4i+1}) - p(2k-t_{4i+2}) + p(2k-t_{4i+3})) = q(k).$$
Then we have
$$\sum_{i=0}^\infty  (p(2k-t_{4i}) - p(2k-t_{4i+1})) =  q(k) + \sum_{i=0}^\infty  (p(2k-t_{4i+2}) - p(2k-t_{4i+3}))$$
which, by Proposition \ref{mexform} and the comment above, can be written  
$$o_1(2k) = q(k) + o_3(2k).$$

Second, suppose $n = 2k+1$.  Write \eqref{ewodd} as
$$\sum_{i=0}^\infty  (p(2k+1-t_{4i}) - p(2k+1-t_{4i+1}) - p(2k+1-t_{4i+2}) + p(2k+1-t_{4i+3})) = 0.$$
Then we have 
$$\sum_{i=0}^\infty  (p(2k+1-t_{4i}) - p(2k+1-t_{4i+1})) =  \sum_{i=0}^\infty  (p(2k+1-t_{4i+2}) - p(2k+1-t_{4i+3}))$$
which, by Proposition \ref{mexform} and the comment above, can be written  
\[  o_1(2k+1) = o_3(2k+1).     \qedhere \]
\end{proof}

The surprising relationship between $o_1(n)$ and $o_3(n)$ allows us to give a short proof of \cite[Thm. 1.2]{an19}.  (Unfortunately, the terms even and odd were reversed in that original theorem statement.)  

\begin{theorem}[Andrews--Newman]
$o(n)$ is almost always even and is odd exactly when $n=j(3j\pm1)$ for some $j$.
\end{theorem}
\begin{proof}
By Proposition \ref{o13}, 
\begin{align*}
o(n) = o_1(n) + o_3(n) 
&= \begin{cases} 2o_3(n) & \text{if $n$ is odd,} \\ 2o_3(n) + q(n/2) & \text{if $n$ is even,} \end{cases} \\
&\equiv  \begin{cases} 0 \bmod 2 & \text{if $n$ is odd,} \\  q(n/2) \bmod 2 & \text{if $n$ is even.} \end{cases} 
\end{align*}
That is, the parity of $o(n)$ for even $n$ reduces to the parity of $q(n/2)$.  The result follows from 
Euler's well-known pentagonal number theorem.
\end{proof}

We conclude with some ideas for further exploration.  Proposition \ref{o13} suggests a bijective proof perhaps along the lines of Franklin's combinatorial proof of Euler's pentagonal number theorem \cite[p. 25]{ae}.  Also, which partitions of $n$ with nonnegative crank correspond to $o_1(n)$ rather than $o_3(n)$?  Similarly, which partitions of $n$ with Frobenius symbols having no 0 in the top row correspond to $o_1(n)$?  Regarding Section \ref{s2}, while there are combinatorial understandings of the crank \cite{bg}, what tools allow for combinatorial proofs involving the mex and $\mex_j$?


\begin{thebibliography}{9}

\bibitem{a}
Andrews, G. E. (2011).  Concave compositions.  \textit{Electron. J. Combin.} 18 P6.

\bibitem{ae}
Andrews, G. E., Eriksson, K. (2004).  Integer Partitions.  Cambridge Univ. Press, Cambridge.

\bibitem{ag}
Andrews, G. E., Garvan, F. G. (1988).  Dyson's crank of a partition.  \textit{Bull. Amer. Math. Soc.} 18 167--171.

\bibitem{an19}
Andrews, G. E., Newman, D. (2019).  Partitions and the minimal excludant. \textit{Ann. Comb.} 23 249--254.

\bibitem{an20}
Andrews, G. E., Newman, D. (2020).  The minimal excludant in integer partitions. \textit{J. Integer Seq.} 23 20.2.3.

\bibitem{ao}
Andrews, G. E., Ono, K. (2005).  Ramanujan's congruences and Dyson's crank.  \textit{Proc. Natl. Acad. Sci. USA} 102 15277.

\bibitem{bg}
Berkovich, A., Garvan, F. G. (2002).  Some observations on Dyson's new symmetries of partitions. \textit{J. Combin. Theory Ser. A} 100 61--93.

\bibitem{d}
Dyson, F. (1944). Some guesses in the theory of partitions. \textit{Eureka} 8 10-15.  

\bibitem{e}
Ewell, J. A. (1973).  Partition recurrences.  \textit{J. Combin. Theory Ser. A} 14 125--127.

\bibitem{f}
Frobenius, F. G. (1900).  \"Uber die Charaktere der symmetrischen Gruppe.  \textit{Sitzungsber. K. Preuss. Akad. Wisse. Berlin} 1900 516--534.

\bibitem{g}
Garvan, F. G. (1988).  New combinatorial interpretations of Ramanujan's partition congruences mod 5, 7, 11. \textit{Trans. Amer. Math. Soc.} 305 47--77.

\bibitem{gr}
Gasper, G., Rahman, M. (2004). Basic Hypergeometric Series, second ed.  Cambridge Univ. Press, Cambridge.

\bibitem{gh}
Gordon, B., Houten, L. (1968). Notes on plane partitions, II. \textit{J. Combin. Theory} 4 81--99. 

\bibitem{hs}
Hopkins, B., Sellers, J. A. (2020).  Turning the partition crank.  \textit{Amer. Math. Monthly} 127 654--657.

\bibitem{s}
Shen, Y. (2019).  On partitions classified by smallest missing part.  \textit{Ramanujan J.} 49: 411--419.

\bibitem{o}
Sloane, N. J. A., ed., (2020). The Online Encyclopedia of Integer Sequences. \url{oeis.org}.


\end{thebibliography}
\end{document}